\documentclass[12pt, a4paper]{article}
\usepackage{titlesec} 
\usepackage{amsfonts}
\usepackage{amsthm}
\usepackage{amssymb}
\usepackage{amsmath}
\usepackage{theoremref}
\usepackage{enumerate}
\usepackage{bbm}
\usepackage{bm}
\usepackage{authblk}
\usepackage{mathtools}
\usepackage{xcolor}

\DeclarePairedDelimiter\abs{\lvert}{\rvert}%
\DeclarePairedDelimiter\norm{\lVert}{\rVert}%

\makeatletter
\let\oldabs\abs
\def\abs{\@ifstar{\oldabs}{\oldabs*}}
\let\oldnorm\norm
\def\norm{\@ifstar{\oldnorm}{\oldnorm*}}
\makeatother

\makeatletter
\g@addto@macro\bfseries{\boldmath}
\makeatother

\newcommand{\A}{\mathcal{A}}

\newcommand{\Mho}{\mathcal{M}^0_{h}}
\newcommand{\Mhdo}{\mathcal{M}^0_{h,d}}
\newcommand{\Mh}{\mathcal{M}_h}
\newcommand{\Mhd}{\mathcal{M}_{h,d}}
\newcommand{\res}{\text{res}_h}
\newcommand{\co}{\text{core}_h}
\newcommand{\resLOG}{\text{res}}
\newcommand{\coLOG}{\text{core}}

\newcommand{\N}{\mathcal{N}}
\newcommand{\T}{\mathbb{T}}

\newcommand{\conj}[1]{\overline{#1}}
\newcommand{\D}{\mathbb{D}}

\newcommand{\Po}{\mathcal{P}}
\newcommand{\cD}{\conj{\mathbb{D}}}

\newcommand{\m}{\textit{m}}
\newcommand{\hil}{\mathcal{H}}

\newcommand{\hb}{\mathcal{H}(b)}

\renewcommand\Re{\operatorname{Re}}

\newtheorem{thm}{Theorem}[section]
\newtheorem{lemma}[thm]{Lemma}

\newtheorem{prop}[thm]{Proposition}

\newtheorem*{thmA}{Theorem A}
\newtheorem*{thmB}{Theorem B}
\newtheorem*{thmC}{Theorem C}
\newtheorem*{corD}{Corollary D}

\theoremstyle{definition}

\theoremstyle{definition}
\newtheorem{remark}[thm]{Remark}
\newtheorem{definition}{Definition}[section]

\newcommand{\Addresses}{{
		\bigskip
		\footnotesize
		
		Bartosz Malman, \\ \textsc{KTH Royal Institute of Technology, \\
			Stockholm, Sweden}\\
			\texttt{malman@kth.se}		
	}}

\begin{document}
\title{\textbf{Thomson decompositions of measures in the disk}}
\author{Bartosz Malman}
\date{ }
\maketitle

\begin{abstract}
We study the classical problem of identifying the structure of $\Po^2(\mu)$, the closure of analytic polynomials in the Lebesgue space $L^2(\mu)$ of a compactly supported Borel measure $\mu$ living in the complex plane. In his influential work, Thomson showed that the space decomposes into a full $L^2$-space and other pieces which are essentially spaces of analytic functions on domains in the plane. For a family of measures $\mu$ supported on the closed unit disk $\cD$ which have a part on the open disk $\D$ which is similar to the Lebesgue area measure, and a part on the unit circle $\T$ which is the restriction of the Lebesgue linear measure to a general measurable subset $E$ of $\T$, we extend the ideas of Khrushchev and calculate the exact form of the Thomson decomposition of the space $\Po^2(\mu)$. It turns out that the space splits according to a certain decomposition of measurable subsets of $\T$ which we introduce. We highlight applications to the theory of the Cauchy integral operator and de Branges-Rovnyak spaces.
\end{abstract}

\section{Introduction and main results}
\label{introsec}
\subsection{Background}

Let $\mu$ be finite positive Borel measure in the complex plane $\mathbb{C}$ which is compactly supported, and consider the usual Lebesgue space $L^2(\mu)$ of equivalence classes of Borel measurable functions which are square-integrable with respect to $\mu$. The set of \textit{analytic polynomials} $\Po$ consists of functions $p$ of the form $p(z) = \sum_{n=0}^N p_n z^n$, where the coefficients $p_n$ are complex numbers, and $z = x+iy$ is the complex variable. The norm closure of $\Po$ in $L^2(\mu)$ is denoted by $\Po^2(\mu)$. The goal of this paper is to determine the structure of $\Po^2(\mu)$ for a family of relatively simple, but important in the applications, measures $\mu$ living in the closed unit disk $\cD := \{ z \in \mathbb{C} : |z| \leq 1 \}$.

We recall two basic examples of such polynomial closures that appear in analysis which turn out to very accurately illustrate what occurs in general.

\begin{enumerate}[(i)]
    \item Let $\mu$ be the restriction to the interval $[0,1]$ of the linear Lebesgue measure $dx$ on the real line $\mathbb{R}$. Then a consequence of the classical Weierstrass polynomial approximation theorem is that the polynomials $p(x) = \sum_{n=0}^N p_nx^n$ are dense in space of square-integrable function on $[0,1]$. Thus $\Po^2(\mu) = L^2(\mu)$ in this case.
    
    \item Instead now, let $\mu = \m$ be the linear Lebesgue measure on the unit circle $\T := \{ z \in \mathbb{C} : |z| = 1 \}$. The corresponding closure $\Po^2(\mu)$ is well-known to not coincide with the full space $L^2(\mu) = L^2(\T)$, but instead with the Hardy space $H^2$, consisting of functions on $\T$ with vanishing negative portion of their Fourier series. The space contains no non-trivial characteristic functions (the trivial ones correspond to the constant functions $0$ and $1$). Each function $f \in H^2 = \Po^2(\mu)$ determines an analytic function in the open unit disk $\D := \{ z \in \mathbb{C} : |z| < 1 \}$, and appropriately defined non-tangential boundary values of this extension coincide with $f$ on $\T$. The magnitude of the values in $\D$ of the extended function is controlled by the size of $f$ on the carrier set of the defining measure, this carrier set being in this case $\T$. More precisely, for each $\lambda \in \D$, there exists a constant $C_\lambda > 0$ such that \begin{equation}
        \label{bpeform} |f(\lambda)| \leq C_\lambda \|f\|_{L^2(\mu)}.
    \end{equation} It is well-known that, in fact, we have $C_\lambda = (1-|\lambda|^2)^{-1/2}$. 
    
    \item Let us now restrict the Lebesgue measure $m$ on $\T$ to proper subset $A$ which is an arc of the circle: $d\mu = 1_A \, d\m$, where $1_A$ is the characteristic function of $A$. In this case, $\Po^2(\mu) = L^2(\mu)$, similarly to $(i)$ above. 
\end{enumerate}

If \eqref{bpeform} holds for some point $\lambda$, then the operation taking $f \in \Po^2(\mu)$ to $f(\lambda)$ is called a \textit{bounded point evaluation}. The operation is called an \textit{analytic bounded point evaluation} if an entire open neighbourhood of $\lambda$ admits a uniform such bound. A long line of research on existence of bounded point evaluations culminated in the work of Thomson in \cite{thomson1991approximation}, who proved that the condition $\Po^2(\mu) = L^2(\mu)$ is the only obstruction to existence of analytic bounded point evaluations for $\Po^2(\mu)$. He proves that for any measure $\mu$, there exists a partition $\mu = \mu_0 + \sum_{n \geq 1} \mu_i$, where each $\mu_i$ is the restriction of $\mu$ to a subset of its support, and a decomposition \begin{equation} \label{thomsondecomp}
    \Po^2(\mu) = L^2(\mu_0) \oplus \Big( \oplus_{i \geq 1} \Po^2(\mu_i)\Big)
\end{equation} such that for $i \geq 1$, $\Po^2(\mu_i)$ contains no non-trivial characteristic functions, and there exists an open simply connected set $U_i$ consisting of analytic bounded point evaluations for $\mu_i$. For $i \geq 1$, the pieces $\Po^2(\mu_i)$ are thus essentially spaces of analytic functions on the domain $U_i$. The deep work of Aleman, Sundberg and Richter in \cite{aleman2009nontangential} complements Thomson's results, and provides important information about the growth, zero sets and boundary behaviour of functions in such \textit{"analytic"} $\Po^2(\mu)$-spaces on the unit disk, i.e., when the corresponding open set of analytic point evaluations is the unit disk $\D$. 

For clarity, we mention that Thomson and Aleman, Richter and Sundberg, prove their main results also in the non-Hilbertian context, with exponents $t \neq 2$ and Lebesgue spaces $L^t(\mu)$. In this paper, we shall focus on the case $t = 2$, to avoid a few rather inessential technical details. We deal with a family of measures of the form \begin{equation} \label{muform1}
d\mu = (1-|z|)^{\alpha} dA + 1_E \, d\m,
\end{equation} where $dA$ is the Lebesgue area measure of the open unit disk $\D$, $d\m$ is the Lebesgue linear measure of the circle $\T$, and $1_E$ is the characteristic (or indicator) function of a general measurable subset $E$ of $\T$. In order for the measure to be finite, we need the restriction $\alpha > -1$. The parameter $\alpha$ will not play an important role, unlike the parameter $E$ which will be at the center of our attention. For this class of measures, we seek the exact form of the decomposition appearing in \eqref{thomsondecomp}. An \textit{"analytic"} piece $\Po^2(\mu_i)$ will appear in the decomposition \eqref{thomsondecomp} of our measure, since the part $\mu_\D := (1-|z|)^{\alpha}dA$ of the measure ensures that $U = \D$ will consist of analytic bounded point evaluations for $\Po^2(\mu)$. The point in question is if a full $L^2$-piece appears in the decomposition, and if so, what is its exact structure. 

\subsection{Work of Khrushchev and our main result}

To motivate and explain our main theorem, we recall the work of Khrushchev in \cite{khrushchev1978problem} which implies some partial results of the type that we are interested in. We assume that $\mu$ has the form \eqref{muform1}, and we will explain the structure of $\Po^2(\mu)$ for different special choices of the set $E$.

\begin{enumerate}[(i)]
\setcounter{enumi}{3}
    \item Assume that $E$ in \eqref{muform1} is an arc. In this case, the space $\Po^2(\mu)$ will contain no $L^2$-summand. To see this, consider the usual Carleson curvilinear rectangle $S := \{ z \in \D : 1-|z| > |E|, z/|z| \in E\},$ and $\phi: \D \to S$ a conformal mapping. A simple computation involving the harmonic measure of $S$ will show that for any sequence of polynomials bounded in the norm of $\Po^2(\mu)$, the sequence $p \circ \phi$ must be bounded in the Nevanlinna class $\N$ of the unit disk. Then the classical Khinchin–Ostrowski theorem (see \cite{havinbook} for precise definitions), which connects the values in $\D$ to values in $\T$ for converging sequences of functions in $\N$, ensures that any convergent sequence of polynomials in $\Po^2(\mu)$ cannot have a non-trivial characteristic function as its limit.
    
    \item Khrushchev in \cite{khrushchev1978problem} and Kegejan in \cite{kegejanex} extend the idea in the previous point to sets more general than intervals. If $E$ is a closed subset of $\T$, and $\T \setminus E = \cup_{n} \ell_n$ is the decomposition of its complement into disjoint open arcs $\ell_n$, then the set $E$ is called a \textit{Beurling-Carleson set} if \begin{equation}
        \label{BCdefintro}
        \sum_{n} |\ell_n|\log(1/|\ell_n|) < \infty.
    \end{equation} For such a set, Khrushchev's approach in \cite{khrushchev1978problem} is to carefully estimate the harmonic measure of a certain domain $\Omega \subset \D$ constructed from the set $E$ and conclude, as before, that any sequence of analytic polynomials bounded in the $\Po^2(\mu)$-norm cannot simultaneously converge pointwise to zero on $\D$ and a non-zero quantity on $\T$ (or vice versa). Thus, if $E$ in \eqref{muform1} is a Beurling-Carleson set, then again $\Po^2(\mu)$ contains no non-trivial characteristic functions (we note that Khrushchev in \cite{khrushchev1978problem} does not work with $\Po^2(\mu)$-spaces, but with certain Banach spaces, but his results are directly applicable to our situation). 
    
    \item Conversely, assume that $E$ in \eqref{muform1} is a measurable subset of $\T$ which \textit{contains no Beurling-Carleson subsets of positive Lebesgue measure}. Under this assumption, Khrushchev constructs a sequence of analytic polynomials that simultaneously converges pointwise to $1$ on $\D$, to $0$ on $E$, and is bounded in the norm of $\Po^2(\mu)$, with $\mu$ given by \eqref{muform1}. It follows that $\Po^2(\mu)$ contains $L^2(E)$ as a summand.
\end{enumerate}
    
Consideration of the above examples leads us to the following definition. Let $E$ be a measurable subset of $\T$, and consider the quantity \begin{equation}\label{CsupremumBC}
C := \sup_{G} \Big\{ |G| : |G \setminus E| = 0, G \text{ is a Beurling-Carleson set } \Big\}.
\end{equation}
That is, we compute the supremum of the Lebesgue measures of Beurling-Carleson sets $G$ which are \textit{almost contained} in $E$, in the sense that the set difference $|G \setminus E|$ has Lebesgue measure zero. Let $\{G_n\}_n$ be a sequence of such Beurling-Carleson sets for which the supermum in \eqref{CsupremumBC} is attained: $\lim_{n \to \infty} |G_n| = C$. We define \begin{equation} \label{coreintrodef}
\coLOG(E) := \cup_n G_n,
\end{equation} and call this set for the \textit{core} of $E$. The set $\coLOG(E)$ is the biggest subset of $E$ which can be synthesized from Beurling-Carleson sets. We define also \begin{equation} \label{resintrodef}
\resLOG(E) := E \setminus \coLOG(E),
\end{equation} which we call for the \textit{residual} of $E$. Then $E$ equals the union of its core and its residual, up to a set $N$ of Lebesgue measure zero. This partition of $E$ will be called for the \textit{core-residual decomposition} of $E$. It is not hard to see that the decomposition is unique, up to sets of Lebesgue measure zero. 

The main result of this paper is that the spaces $\Po^2(\mu)$ under consideration split up according to the core-residual decomposition.

\begin{thmA} Let $d\mu = (1-|z|)^{\alpha}dA + 1_E \,d\m$, where $\alpha > -1$ and $E$ is any measurable subset of $\T$. Let \begin{equation} \label{dmu1def}
    d\mu_1 := d\mu - 1_{\resLOG(E)}d\m = (1-|z|)^{\alpha}dA + 1_{\coLOG(E)}\, d\m.
\end{equation} Then 
\begin{equation} \Po^2(\mu) = L^2(\resLOG(E)) \oplus \Po^2(\mu_1),
\end{equation} where $\Po^2(\mu_1)$ contains no non-trivial characteristic functions.
\end{thmA}

By $L^2(\resLOG(E))$ above we mean the Lebesgue space of square-integrable functions on the set $\resLOG(E)$. It should be interpreted as the zero space if $E$ is residual-free: $|\resLOG(E)| = 0$. A consequence of the fact that $\Po^2(\mu_1)$ contains no non-trivial characteristic functions is that every function $f \in \Po^2(\mu_1)$ has an analytic restriction to $\D$, and this restriction is a one-to-one mapping from $\Po^2(\mu_1)$ to analytic functions in $\D$. 

Our result here is a generalization of Khrushchev's (when specialized to the Hilbertian setting). Indeed, the two examples $(v)$ and $(vi)$ above correspond, more or less, to the cases $|\coLOG(E)| = |E|$ and $|\resLOG(E)| = |E|$, respectively. In the proofs, we do use Khrushchev's general strategy, so our approach to his results is not new. Rather we extend his ideas to the general situation in which the residual and the core of the set $E$ are simultaneously of positive measure. 

We should make a comment on the nature of the set $\resLOG(E)$ to better explain the difficulty that arises in our setting. Indeed, this set has the property found in example $(vi)$ above: it contains no Beurling-Carleson subsets of positive Lebesgue measure. This property of a set $F$ alone allows Khrushchev to construct a sequence of polynomials which simultaneously converges to 1 on $\D$, to 0 on $F$ and is bounded in the $\Po^2(\mu)$-norm if $F = E$ in \eqref{muform1}. But the case $F$ being a proper subset of $E$ is different, and the mentioned property of $F$ alone (i.e., containing no Beurling-Carleson subsets of positive measure) cannot in general single-handedly allow for the existence of a non-trivial $L^2(F)$-summand in the space $\Po^2(\mu)$. Indeed, consider a rather banal example $E = \T = F \cup (\T \setminus F)$. In this case,  $\Po^2(\mu)$ certainly contains no $L^2$-summands, and $\Po^2(\mu) = H^2$ as sets, with equality of norms. This example highlights the perhaps obvious fact that $F$ being a residual set is \textit{not} an intrinsic property of $F$, but is tightly related to the surrounding set $E$ that the residual is produced with respect to. The \textit{"splitting off"} of the piece $L^2(\resLOG(E))$ from $\Po^2(\mu)$ is in fact related to the geometric way in which $\resLOG(E)$ is contained in $E$, see \thref{mainCRestimate} below and the remark following it.

The theorem presented here is a specialization to the simplest and most important case of our most general result, which is \thref{p2musplittheorem} below. For most of the paper, we adopt a convention similar to Khrushchev's in \cite{khrushchev1978problem} and work in the more general case of sets which we call \textit{$h$-Carleson} (where \eqref{hcarlesondef} replaces the definition \eqref{BCdefintro} presented above), generalized core-residual decompositions, and measures $\mu$ which are compatible with the new definition. The generalization is not deeper than replacing the function $t \log(1/t)$, which is doing the work in \eqref{BCdefintro}, by any function $h(t)$ with similar properties. This change adds no technical troubles, and we would say that it rather helps to emphasize the essential parts in the proofs.

\subsection{Application I: functions with sparse support and strong unilateral decay of Fourier coefficients}

One of Khrushchev's main applications of his results, which we described in the examples $(v)$ and $(vi)$ above, was in the study of the existence question for the following type of functions. Consider $f(\zeta) = \sum_{n \in \mathbb{Z}} f_n\zeta^n \in L^2(\T)$ which vanishes on the measurable subset $\T \setminus E$; \begin{equation}
    \label{fsupportappl}
    f(\zeta) = 0, \text{ for almost every } \zeta \in \T \setminus E,
\end{equation} and yet satisfies a unilateral rapid decay condition on its Fourier coefficients: 
\begin{equation} \label{coefficientdecay}
    \lim_{n \to +\infty} |f_n|n^K = 0, \quad \text{ for all } K \geq 0.\end{equation}
In other words, we look at of those square-integrable functions $f$ which live only on $E$, and their \textit{Cauchy transforms} 
\begin{equation} \label{Cfeq}
C_f(z) := \int_{\T} \frac{f(\zeta)}{1-z\conj{\zeta}} d\m(\zeta) = \sum_{n \geq 0} f_nz^n, \quad z \in\D   \end{equation}
are members of the algebra $\A^\infty$ of analytic functions in $\D$ which extend continuously to $\cD$, and so do their derivatives of any order. Let us define 
\begin{equation}
    \label{Splusdef}
    S_+(E) := \Big\{ f \in L^2(E) : \eqref{coefficientdecay} \text{ holds } \Big\}.
\end{equation} 
It is not hard to see that a bilateral version of \eqref{coefficientdecay} cannot hold for all possible sets $E$. A Fourier series with such bilateral rapid decay defines a function of the class $C^\infty$, and a closed set $E$ which is nowhere dense in $\T$ cannot support even a continuous function on $\T$. Khrushchev found, by non-constructive methods, that $S_+(E)$ admits a non-zero function essentially if and only if $E$ contains a Beurling-Carleson subset of positive Lebesgue measure. He showed also that if $E$ contains no Beurling-Carleson subsets of positive Lebesgue measure, then $S_+(E) = \{0\}$. 

In \cite{ConstrFamiliesSmoothCauchyTransforms}, a construction of a non-zero function $f \in S_+(E)$ has been presented, whenever $E$ is a Beurling-Carleson set. From that construction it follows easily that $S_+(E)$ is not only non-trivial, but even norm-dense in $L^2(E)$, whenever $E$ is a Beurling-Carleson set. A natural follow-up is to try to characterize the measurable subsets $E$ of $\T$ for which the set $S_+(E)$ is norm-dense in $L^2(E)$. The investigation carried out in the present paper leads us to the following conclusion.

\begin{thmB}
Let $E$ be a measurable subset of $\T$ of positive Lebesgue measure. Then the norm-closure of the set $S_+(E)$ coincides with $L^2(\coLOG(E))$. Hence the set $S_+(E)$ is dense in $L^2(E)$ if and only if $|\resLOG(E)| = 0$.
\end{thmB}

In fact, the condition \eqref{coefficientdecay} can be very significantly weakened, and theorem as stated above still holds true even for this weaker condition. For instance, if $|\resLOG(E)| > 0$, then even functions $f$ living on $E$ for which the Dirichlet integral of $C_f$ is finite, will not be dense in $L^2(E)$. As usual, if $g$ is analytic in $\D$, then its Dirichlet integral is the quantity \begin{equation} \label{Dirichletintegral}
\int_\D |g'(z)|^2 \, dA(z).
\end{equation}  We will establish this stronger versions in the proofs below. 

Our proof of this stronger version, which is presented in Section \ref{proofTBsec}, has the following curious consequence. Assume that we have a function $f \in L^2(\T)$, and we have some information regarding where $f$ lives: \begin{equation}
    \label{Appl1fcarrier}
    f(\zeta) = 0, \text{ for almost every } \zeta \in \T \setminus E,
\end{equation} Assume further that we can cut the Fourier series of the function $f$ into two pieces with semi-bounded spectrum $$f(\zeta) = \sum_{n < N} f_n \zeta^n + \sum_{n \geq N} f_n\zeta^n := f_-(\zeta) + f_+(\zeta)$$ and measure that one of the pieces, say $f_+$, has \textit{finite logarithmic energy} (see \cite{havinbook}). Namely, 
\begin{equation}
    \label{Appl1f2finenergy}
    \sum_{n \geq N} n|f_n|^2 < \infty.
\end{equation}
In that case, we can improve our information on the localization of $f$ to 
\begin{equation}
\label{Appl1fcarrierImproved}
    f(\zeta) = 0, \text{ for almost every } \zeta \in \T \setminus \coLOG(E).
\end{equation}

\subsection{Application II: approximations in $\hb$-spaces}

Theorem B has a different interpretation in the context of the de Branges-Rovnyak spaces $\hb$. The space $\hb$ can be defined as the unique Hilbert space of analytic functions on $\D$ with a reproducing kernel of the form \begin{equation}
    \label{hbkernel} k_b(z,\lambda) = \frac{1-\conj{b(\lambda)}b(z)}{1-\conj{\lambda}z}, \quad z,\lambda \in \D.
\end{equation} Here $b: \D \to \D$ is an analytic self-map of the unit disk $\D$. This class of spaces is important for its connections to the model theory of general linear operators on Hilbert space, and also for its various connections to complex function theory. One can consult the monographs \cite{sarasonbook}, \cite{hbspaces1fricainmashreghi} and \cite{hbspaces2fricainmashreghi} for more background.

It is in general hard to give explicit formulas for any inhabitants of a given space $\hb$, except for the kernel functions in \eqref{hbkernel}. Thus, of special importance are results which explain what various functions enjoying some additional regularity properties can be found in a given space $\hb$, and to what extend these functions can be used to approximate a general $\hb$-function. With this in mind, we present a version of Theorem B which is specialized to the context of de Branges-Rovnyak spaces. The set $E$ in the statement below is defined, up to a subset of Lebesgue measure zero, by the usual non-tangential boundary values of the analytic function $b$.

\begin{thmC}
Let $b: \D \to \D$ be analytic, and consider the measurable set $$E := \{ \zeta \in \T : |b(\zeta)| < 1 \}.$$ A necessary condition for the functions with finite Dirichlet integral in \eqref{Dirichletintegral} to be norm-dense in $\hb$ is that $|\resLOG(E)| = 0$.
\end{thmC}

The papers \cite{DBRpapperAdem} and \cite{ConstrFamiliesSmoothCauchyTransforms} deal with positive approximation results. It is found, for instance, that if in addition to the condition $|\resLOG(E)| = 0$, the function $1-|b(\zeta)|$ is not too small on $E$, and the inner factor of $b$ satisfies certain properties, then the set $\A^\infty \cap \hb$ will be dense in $\hb$. In total thus, three conditions are presented in \cite{ConstrFamiliesSmoothCauchyTransforms} which imply this density statement. By consideration of examples in \cite{DBRpapperAdem}, these three conditions are presumed to be close to optimal. Theorem C confirms that one of these conditions, the vanishing of the residual set above, is in fact necessary. 

\subsection{A comment on the weighted circle case}

A more general weighted variant of our problem, in which the characteristic function $1_E$ in \eqref{muform1} is replaced by a non-negative measurable function $w$ living on some carrier set $E$, is more difficult. In this setting, it is both the structure of the carrier set $E$ and the local smallness of the weight $w$ that cause $L^2$-summands to appear in the Thomson decomposition of the measure. 

Significant efforts have been made to understand the conditions under which the space $\Po^2(\mu)$ contains $L^2(w \, d\m)$, which of course is the biggest possible non-analytic summand. The curious reader can consult \cite{kriete1990mean}, in which a smallness criterion on the weight $w$ is shown to imply this \textit{complete splitting} of the measure into an analytic piece on $\D$ and a full $L^2(w\, d\m)$-piece on $\T$. Interesting methods are employed in \cite{kriete1990mean} which are different from ours, and rely on estimates of certain composition operators. Our present work has a different focus, and it is the identification of \textit{any} $L^2$-summands inside of $\Po^2(\mu)$ that is our principal interest. In the weighted context, we deduce from our results that $\Po^2(\mu)$ contains an $L^2$-summand if the carrier of the weight $w$ has a non-trivial residual.

\begin{corD}
Let $\mu$ have the form $$d\mu = (1-|z|)^{\alpha} + w \,d \m,$$ where $\alpha > -1$ and $w$ is an integrable measurable non-negative weight on $\T$. Consider the measurable set $$E := \{ \zeta \in \T : w(\zeta) > 0 \}$$ and let $1_{\resLOG(E)} \cdot w$ be the restriction of the weight $w$ to the residual set of $E$. Then $L^2(1_{\resLOG(E)} \cdot w \,d\m)$ is contained in $\Po^2(\mu)$.
\end{corD}

The corollary is a direct consequence of \thref{p2musplittheorem} below, from which Theorem A will also be deduced.

\subsection{Outline of the paper}

In Section \ref{hausdorffsec} we recall some properties of the Hausdorff measures (or more precisely, Hausdorff contents), and alter their definitions a bit to produce Hausdorff-type measures which are more suitable to our purposes. The next part, Section \ref{CRsection}, deals with the definition and basic properties of the core-residual decompositions. Here, the main result is \thref{mainCRestimate}, which connects the decomposition to the Hausdorff-type measures defined earlier. The technical constructions of main objects occur in Sections \ref{frostmansection} and \ref{realvariablessection}, where we specialize some ideas from \cite{khrushchev1978problem}. The proof of the main result, Theorem A, occurs in Sections \ref{splitsection} and \ref{biggestsummandsection}. The last two Sections \ref{proofTBsec} and \ref{proofTCsec} deal with proofs of Theorems B and C.

\subsection{Conventions and notation}

We identify the circle $\T$ and the interval $[0,2\pi)$ by the usual mapping. An arc $A$ of the circle is sometimes called for an interval, and the notation $A = (a,b)$ means that $A$ is the shorter of the two circular arcs with endpoints $a,b \in \T$. For subsets $A$ of $\T$, we use the notation $|A| := m(A)$ and a normalization so that $m(\T) = 1$. The symbol $L^2(A)$ denotes the usual Lebesgue space of measurable functions living only on $A$ and square-integrable with respect to the restriction of the Lebesgue linear measure to $A$. That is, $L^2(A) = L^2(1_A \cdot m)$. For a measure $\mu$, the notations $\mu_\D$ and $\mu_\T$ denote the restrictions of the measure to $\D$ and $\T$, respectively. We shall sometimes take complements of a closed set $A$ with respect to an interval $I$. This complement is $I \setminus A$. Strictly speaking this might not be open (considered as a subset of $\T$) if $I$ is closed, since two of the complementary arcs might be only half-open. In this case, being \textit{open} means \textit{relatively open} with respect to $I$. 

\section{Hausdorff-type measures}

\label{hausdorffsec}

We begin our discussion by reviewing the notion of Hausdorff measures (or Hausdorff contents) and their basic properties. We will introduce also some useful variations of the classical definition. 

\subsection{Measure functions}

Our Hausdorff-type measures will be defined using a measuring function $h$. The functions which we consider in our discussions satisfy the following basic assumptions:

\begin{enumerate}[(i)]
    \item $h$ is defined for real $t \in [0, \delta)$, where $\delta$ is some positive number, and it is increasing and continuous, with $h(0) = 0$,
    \item $h(t)/t$ is decreasing, with $\lim_{t \to 0^+} h(t)/t = +\infty$.
\end{enumerate}

Any function satisfying the above requirements will be called a \textit{measure function}. We shall not make a fuss about the size of $\delta > 0$ in $(i)$ above, so we assume that it is big enough so that all our computations are valid. Any function $h$ satisfying $(i)$ and $(ii)$ can be modified or extended to a function on $(0, \infty)$ with the same properties, in such a way that the its values near $0$ remain unchanged. Such a modified function induces equivalent Hausdorff-type measures.

An immediate consequence of $(ii)$ above is subadditivity:
\begin{gather*}
h(t+s) = (t+s)\frac{h(t+s)}{t+s} = t\frac{h(t+s)}{t+s} + s\frac{h(t+s)}{t+s} \\ \leq t\frac{h(t)}{t} + s\frac{h(s)}{s} = h(s)+h(t).
\end{gather*}

We will develop our most comprehensive results for a general measure function $h$, but the main example  we have in mind is $$h(t) = t \log(e/t), \quad t \in (0,1],$$  which we extend by continuity to satisfy $h(0) = 0$. The constant $e$ appearing in the expression ensures that the function is increasing on $[0,1]$.

\subsection{Hausdorff measures}

For a measure function $h$ satisfying the above properties, we recall the definition of the Hausdorff measure $\Mh$ of a subset $E$ of the unit circle, given by the expression

\begin{equation}\label{Mhdef}
\Mh(E) := \inf_L \Big\lbrace \sum_{\ell \in L} h(|\ell|) : E \subset \cup_{\ell \in L}, \, \ell \text{ open interval} \Big\rbrace.
\end{equation} That is, $\Mh(E)$ is the infimum of the numbers $\sum_{\ell \in L} h(|\ell|)$, taken over covers $L$ of $E$. By continuity of $h$, in fact it matters not if we take the intervals $\ell$ to be open, closed, or half-open, and sometimes it is convenient to assume that the intervals are one, or the other.

We define also a measure where the covering families $L$ above are restricted to consist of dyadic intervals. Let \begin{equation} \label{dyadicintervals}
d_{n,j} = \{ e^{it} : t \in [2\pi2^{-n}j, 2\pi2^{-n}(j+1)] \}, \quad n \in \mathbb{N}, \, 0 \leq j < 2^n - 1,
\end{equation}
\begin{equation}
\label{Dnintervals}
    \mathcal{D}_n := \{d_{n,j}\}_{j=1}^{2^{n-1}}
\end{equation} and 
\begin{equation*}
    \mathcal{D} = \cup_{n \in \mathbb{N}} \mathcal{D}_n.
\end{equation*}  
The set $\mathcal{D}_n$ consists of the \textit{dyadic intervals of $n$th generation}, and $\mathcal{D}$ is the family of dyadic intervals covering $\T$. We define
\begin{equation}\label{Mhddef}
\Mhd(E) := \inf_L \Big\lbrace \sum_{d \in L} h(|d|) : E \subset \cup_{d \in L}, \, d \in \mathcal{D}\Big\rbrace.
\end{equation}
The two quantities $\Mh(E)$ and $\Mhd(E)$ are comparable. 
\begin{prop} \thlabel{mhequiv}
For any set $E \subset \T$ we have $$\Mh(E) \leq \Mhd(E) \leq 2\Mh(E).$$
\end{prop}
\begin{proof}
The family of admissible covers $L$ is more restrictive in the definition of $\Mhd$, thus $\Mh(E) \leq \Mhd(E).$ On the other hand, if $L$ is any cover of $E$ consisting of any kind of intervals, then each $\ell \in L$ can be covered by two dyadic intervals $d_1, d_2$ of equal length $|d_1| = |d_2|$ which satisfy $|d_1| < |\ell| \leq 2|d_1|$. This implies, by the postulated properties of $h$, that $$h(|\ell|) \leq h(|d_1|) + h(|d_2|) \leq 2 h(|\ell|).$$ Summing over all $\ell \in L$ and passing to the infimum over all covers $L$, we obtain $\Mhd(E) \leq 2\Mh(E).$
\end{proof}

The measure function $\Mhd$ has the following crucial continuity property. For a proof, we refer the reader to \cite[page 301]{havinbook} or \cite[page 9]{carleson1967selected}

\begin{prop} \thlabel{Mhdcont}
If $A_n$ is an increasing sequence of subsets of $\T$ for which $A = \cup_n A_n$ holds, then $$\Mhd(A) = \lim_{n \to \infty} \Mhd(A_n).$$
\end{prop}

\subsection{Small Hausdorff measures}

We will also find it convenient to work with the following versions of $\Mh$ and $\Mhd$ which disregard sets of Lebesgue measure zero. We define the quantity $\Mho(E)$ analogously to \eqref{Mhdef}, but the family $\{\ell\}_{\ell \in L}$ is only required to cover $E$ up to a subset of Lebesgue measure zero. Hence 

\begin{equation} \label{mhzerodef}
\Mho(E) := \inf_L \Big\lbrace \sum_{\ell \in L} h(|\ell|) : |E \setminus (\cup_{\ell \in L} \ell)| = 0, \, \ell \text{ open interval} \Big\rbrace.
\end{equation}
and 
\begin{equation} \label{mhdzerodef}
\Mhdo(E) := \inf_L \Big\lbrace \sum_{d \in L} h(|\ell|) : |E \setminus (\cup_{d \in L} d)| = 0, \, d \in \mathcal{D} \Big\rbrace.
\end{equation}

A version of \thref{mhequiv} obviously also holds for the measures $\Mho$ and $\Mhdo$. Other simple relations between our Hausdorff-type measures are the following.

\begin{prop} \thlabel{mhinequalities}
For any measurable subset $E$ of $\T$ it holds that 

$$\Mh(E) \geq \Mho(E) \geq h(|E|)$$ and
$$\Mhd(E) \geq \Mhdo(E) \geq h(|E|).$$ 
\end{prop}

\begin{proof}
The "small" versions $\Mho$ and $\Mhdo$ allow for more general covers $L$, so the first inequalities in both lines are clear. Moreover, if $|E \setminus (\cup_{\ell \in L} \ell)| = 0$ then certainly $|E| \leq \sum_{\ell \in L} |\ell|$, and by subadditivity of $h$ we have $$h(|E|) \leq \sum_{\ell \in L} h(|\ell|).$$ By taking infimum over admissable families $L$ we obtain the right-most inequalities above.
\end{proof}

\section{Cores, residuals, and their properties}

\label{CRsection}

In this section we introduce our principal tool, which we call the \textit{core-residual decomposition}. We fix a measure function $h$ satisfying the properties stated in Section \ref{hausdorffsec}. 

\subsection{$h$-Carleson sets}

The following definition is classical.

\begin{definition}
A closed subset $A$ of $\T$ is called an \textit{$h$-Carleson} set if the system $L = \{\ell_n\}_n$ of disjoint open arcs which constitute the complement $\T \setminus A$ satisfies \begin{equation}
    \label{hcarlesondef} 
    \sum_n h(|\ell_n|) < \infty.
\end{equation}
\end{definition}
For completeness, we spend a moment to recall the basic properties of the $h$-Carleson sets which are needed in the following discussion. A simple first observation which one can make is that if we start with a family of not necessarily disjoint open arcs which satisfy \eqref{hcarlesondef}, then the complement of their union is an $h$-Carleson set.

\begin{prop}\thlabel{constrhCarleson}
Let $L = \{\ell\}_{\ell \in L}$ be a family of open intervals of $\T$ for which we have $$\sum_{\ell \in L} h(|\ell|) < \infty.$$ Then $A = \T \setminus (\cup_{\ell \in L} \ell)$ is an $h$-Carleson set.
\end{prop}

\begin{proof}
The set $A$ is obviously closed. Its complement is a union $$\T \setminus A = \cup_n \alpha_n$$ of disjoint open arcs $\alpha_n$, and each $\alpha_n$ is the union of those arcs $\ell \in L$ which intersect it non-trivially. Thus $|\alpha_n| \leq \sum_{|\ell \cap \alpha_n| > 0} |\ell|,$ which by subadditivity of $h$ implies $$h(|\alpha_n|) \leq \sum_{|\ell \cap \alpha_n| > 0} |h(|\ell|).$$
Whenever $\ell$ intersects $\alpha_n$ non-trivially, then $\ell$ is contained in $\alpha_n$. Thus each $\ell \in L$ appears in the sum above only for one distinct interval $\alpha_n$. Summing the above inequality over all arcs $\alpha_n$ proves the claim. 
\end{proof}

The class of $h$-Carleson sets is closed under taking finite unions and intersections.

\begin{prop} \thlabel{carlesonunionprop}
Let $A$ and $B$ be $h$-Carleson sets. Then their union $A \cup B$ and their intersection $A \cap B$ are also $h$-Carleson sets.
\end{prop}

\begin{proof}
Let $\T \setminus A = \cup_{n} \alpha_n$, where $\alpha_n$ are pairwise disjoint open arcs, and $\T \setminus B = \cup_m \beta_m$ be the corresponding decomposition of $\T \setminus B$. The complement $T \setminus ( A \cap B)$ is a disjoint union of open arcs which are themselves unions of arcs $\alpha_n$ and $\beta_n$. So the fact that $A \cap B$ is $h$-Carleson follows from \thref{constrhCarleson}.

For the case of the union, we have $$\T \setminus (A \cup B) = \cup_{n,m} \alpha_n \cap \beta_m,$$ and in order to verify that $A \cup B$ is an $h$-Carleson set, we need to show that $$\sum_{n,m} h(|\alpha_n \cap \beta_n|) < \infty.$$  Let us forget about the indices $(n,m)$ such that $\alpha_n \cap \beta_m$ is void, they contribute nothing to the sum above. We partition the rest of the indices $(n,m)$ into three families. The family $F_1$ consists of those $(n,m)$ such that $\alpha_n \subset \beta_m$, the family $F_2$ consists of those $(n,m)$ such that $\beta_m \subset \alpha_n$, and family $F_3$ of those $(n,m)$ such that $\alpha_n \cap \beta_m$ is non-empty, but neither $\alpha_n$ or $\beta_m$ is contained in the other. Our task is then to show that $$\sum_{(n,m) \in F_i} h(|\alpha_n\cap \beta_m|) < \infty$$ for $i = 1,2,3$. 

We start with $F_1$. If $(n,m) \in F_1$, then there is no other index $m'$ such that $(n,m') \in F_1$, for that would imply $\beta_m \cap \beta_{m'}$ being non-trivial. It follows that each integer $n$ appears at most once as a first coordinate of a pair in $F_1$. Moreover, clearly $|\alpha_n \cap \beta_m| = |\alpha_n|$, and so $$\sum_{(n,m) \in F_1} h(|\alpha_n \cap \beta_m|) = \sum_{(n,m) \in F_1} h(|\alpha_n|)\leq \sum_{n} h(|\alpha_n|) < \infty,$$ by the assumption that $A$ is $h$-Carleson. The family $F_2$ is handled symmetrically.

When it comes to family $F_3$, note that any integer $n$ appears as a first coordinate in at most two pairs in $F_3$. Indeed, if $(n,m) \in F_3$, then it follows that either the left or the right endpoint of $\alpha_n$ is contained in $\beta_m$. Since $\alpha_n$ obviously has only two endpoints, there can be at most two pairs in $F_3$ with first coordinate $n$. Thus $$\sum_{(n,m) \in F_3} h(|\alpha_n \cap \beta_m|) \leq \sum_{(n,m) \in F_3} h(|\alpha_n|) \leq 2 \sum_n h(|\alpha_n|) < \infty.$$ By summing over all the three families we see that $A \cup B$ is an $h$-Carleson set. \end{proof}

\subsection{Core-residual decomposition}

With the definition and basic properties of $h$-Carleson sets in place, we can introduce our decomposition of measurable sets. The decomposition is very simple, but it has a crucial relation to the small Hausdorff-type measures. The relation is stated in \thref{mainCRestimate}.

Here and throughout the paper, if $|A \setminus B| = 0$ holds between two measurable subsets $A$ and $B$, we say that $A$ is \textit{almost contained} in $B$. 

\begin{prop} \thlabel{Edecomp}
Let $E$ be any measurable subset of $\T$ and $h$ a measure function. There exists a decomposition $$E = (F \cup \Tilde{E}) \setminus N,$$ where $N$ is a set of Lebesgue measure zero, and the disjoint measurable subsets $F$ and $\Tilde{E}$ satisfy the following two properties:

\begin{enumerate}[(i)]
    \item if $A$ is an $h$-Carleson set which is almost contained in $E$, then it is almost contained in $\Tilde{E}$,
    \item $\Tilde{E}$ is, up to a set of Lebesgue measure zero, a countable union of $h$-Carleson sets almost contained in $E$.
\end{enumerate} The decomposition is unique, up to a difference of sets of Lebesgue measure zero.
\end{prop}

\begin{proof}
We let the non-negative constant $C$ be the supremum of Lebesgue measures of $h$-Carleson sets which are almost contained in $E$. If $C = 0$ then we are done by taking $E = F$. In the case $C > 0$, let $\{E_n\}_n$ be a sequence of $h$-Carleson sets almost contained in $E$ such that $\lim_{n \to \infty} |E_n| = C$. By \thref{carlesonunionprop}, we can assume that $E_n \subset E_{n+1}$. Take $\Tilde{E} := \cup_n E_n$, and $F = E \setminus \Tilde{E}$. Clearly we have $|\Tilde{E}| = C$, and property $(ii)$ holds.

Assume that $A$ is an $h$-Carleson set almost contained in $E$. Then $A \cup E_n$ is also an $h$-Carleson almost contained in $E$, by \thref{carlesonunionprop}. We have $$C = \lim_{n \to \infty} |E_n| \leq \lim_{n \to \infty} |A \cup E_n|  = |A \cup \Tilde{E}|.$$ However, by definition of $C$ we have that $\lim_{n \to \infty} |A \cup E_n| \leq C$, and so we infer that $$|A \cup \Tilde{E}| = |\Tilde{E}|.$$ Thus $A$ is almost contained in $\Tilde{E}$.

We establish the uniqueness of the decomposition. Let $E = F_1 \cup \Tilde{E}_1$ be some other partition of $E$ where the above stated properties hold and $\Tilde{E}_1$ is (up to Lebesgue measure zero) a union $\cup_n G_n$ of some increasing sequence of $h$-Carleson sets $G_n$ almost contained in $E$. It follows readily that the $h$-Carleson sets $E_n \cup G_n$ are almost contained in the intersection $\Tilde{E} \cap \Tilde{E}_1.$ But then $$|\Tilde{E}_1| = \lim_n |G_n| \leq \lim_n |G_n \cup E_n| \leq |\Tilde{E}_1 \cap \Tilde{E}|,$$ so we must have $\Tilde{E}_1 \subset \Tilde{E}$ up to a set of Lebesgue measure zero. By reversing the argument we see that $\Tilde{E} = \Tilde{E}_1$, up to a set of Lebesgue measure zero, and so the uniqueness of the decomposition is established.
\end{proof}

\begin{definition}
Let $E$ be any measurable subset of $\T$, and consider any decomposition of the form $E = (F \cup \Tilde{E}) \setminus N$ as in \thref{Edecomp}. We will say that $F$ and $\Tilde{E}$ form the \textit{core-residual decomposition} of $E$ with respect to $h$. We will use the notation $$F = \res(E)$$ and $$\Tilde{E} = \co(E),$$ and call the two pieces for the \textit{residual} and the \textit{core}, respectively, of the measurable subset $E$. 
\end{definition}

Note that $\res(E)$ and $\co(E)$ are defined, up to a set of Lebesgue measure zero, for any measurable subset $E$ of $\T$. The set $N$ with $|N| = 0$ will play no role in our development, so we might without loss of generality assume that $N$ is empty, by altering the set $E$ on a set of Lebesgue measure zero.

The core-residual decomposition does not change if we pass to a portion of $E$ contained in an interval, which is a fact that we will often implicitly use.

\begin{prop}
Let $E$ be a measurable subset of $\T$ and $I$ be an interval of $\T$. Then $$\res(I \cap E) = I \cap \res(E)$$ and $$\co(I \cap E) = I \cap \co(E).$$
\end{prop}

The above claim follows readily from the definition of the decomposition and the intersection statement in \thref{carlesonunionprop}. We skip the straight-forward proof. 

The next proposition is also quite simple, but it is our main one in this section. It showcases a connection between the core-residual decompositions and the Hausdorff-type measures introduced in Section \ref{hausdorffsec}. A prototype version appears in \cite[page 303]{havinbook}

\begin{prop}
\thlabel{mainCRestimate}
Let $E$ be any measurable subset of $\T$, and $I$ an interval. Assume that $C$ is a measurable subset of $I$ which is almost contained in $E$. Then $$\Mho\Big(I \setminus C\Big) = \Mho \Big(I \setminus (C \cap \co(E))\Big).$$
\end{prop}

\begin{proof}
We have the set inclusion $I \setminus C \subseteq I \setminus (C \cap \co(E))$, so the inequality $$\Mho\Big(I \setminus C\Big) \leq \Mho \Big(I \setminus (C \cap \co(E))\Big)$$ is plain. To see the converse inequality, assume that $L$ is a family of open intervals for which we have $$|(I \setminus C) \setminus (\cup_{\ell \in L} \ell)| = 0$$ and \begin{equation}\label{Dcarleson} \sum_{\ell \in L} h(|\ell|) \leq \Mho(I \setminus C) + \epsilon,\end{equation} for some $\epsilon > 0$. Let $A$ be the complement of $\cup_{\ell \in L} \ell$ with respect to $I$. Then $A$ is closed, and by \thref{constrhCarleson} and \eqref{Dcarleson}, it is an $h$-Carleson set. Moreover, clearly $A$ is almost contained in $C$, which is almost contained in $E$ by hypothesis. By the definition of the core-residual decomposition, it follows that $A$ is almost contained in $C \cap \co(E)$, and taking complements again, it follows that $I \setminus (C \cap \co(E))$ is almost contained in $\cup_{\ell \in L} \ell$. Thus, the family $L$ is an admissible "cover" of $I \setminus (C \cap \co(E))$ for the measure $\Mho$. Consequently, we have $$\Mho(I \setminus (C \cap \co(E)) \leq \sum_{\ell \in L} h(|\ell|) \leq \Mho(I \setminus C) + \epsilon.$$ The proof of the proposition is complete by letting $\epsilon$ tend to zero.
\end{proof}

\begin{remark} To make some sense of the above result, take $C = E$ in the statement above. Then the conclusion is that the $\Mho$-size of what remains from the interval $I$ after removing $E$, is in fact the same as the $\Mho$-size of what remains after removing only $\co(E)$. Thus, in a sense, the set $\res(E)$ is invisible for the Hausdorff-type measure $\Mho$.  \end{remark}

\section{A measure construction of Frostman}

\label{frostmansection}

In the next result, we extend slightly the classical construction of the Frostman measures (refer to \cite[page 302]{havinbook} or \cite[Appendix D]{garnett2005harmonic} for details). The properties $(ii)$ and $(iii)$ in the proposition below appear classically as features of the Frostman measures, and we obtain a useful additional feature $(i)$ through a simple averaging process.

\begin{prop} \thlabel{frostmanmeasureprop}
Let $U = \cup_n \alpha_n$ be an open subset of $\T$, where $\{\alpha_n\}_n$ is a disjoint family of open intervals. There exists a positive absolutely continuous Borel measure $\nu$ which satisfies the following three properties:

\begin{enumerate}[(i)]
    \item $\nu= \sum_{n=1}^N c_n 1_{\alpha_n}$, where $N$ is some finite integer, and $c_n$ are positive numbers,
    \item $\nu(\Delta) \leq h(|\Delta|)$ for each interval $\Delta$ of $\T$,
    \item $\nu(\T) \geq \Mhd(U)/24$.
\end{enumerate}
\end{prop}

\begin{proof}
By \thref{Mhdcont} there exist a positive integer $N$ and closed intervals $\{d'_n\}_{n=1}^N$ such that $d'_n \subset \alpha_n$, and if $D = \cup_{n=1}^N d'_n$, then $\Mhd(D) \geq \Mhd(U)/2$. The usual Frostman measure construction (see, for instance, \cite[page 302]{havinbook}) will produce a measure $\tilde{\nu}$ which is of the form $\tilde{\nu} = \sum_{n=1}^N \tilde{\nu}_n$, with $\tilde{\nu}_n$ supported on a closed interval $d_n$ living inbetween $d'_n$ and $\alpha_n$, which satisfies $\tilde{\nu}(\Delta) \leq 4 h(|\Delta|)$ for any interval $\Delta \subset \T$, and $\tilde{\nu}(D) \geq \Mhd(D)$. Since $D \subset U$, we also have $\tilde{\nu}(U) = \tilde{\nu}(D) \geq \Mhd(D) \geq \Mhd(U)/2$. Now we replace the measure $\tilde{\nu}$ by its average over every $\alpha_n$. That is, we take $c_n = \tilde{\nu}_n(\alpha_n)|\alpha_n|^{-1}$, and we define \begin{equation}
    \label{muaveragedef} \nu := \sum_{n=1}^N c_n 1_{\alpha_n}. 
\end{equation} 
It is obvious that $\nu(U) = \tilde{\nu}(U)$, and even that $\nu(\alpha_n) = \tilde{\nu}(\alpha_n)$. Moreover, we have not destroyed the smoothness of the measure either, as we will show that $\nu(\Delta) \leq 12 h(|\Delta|)$ for any interval $\Delta$. To this end, fix $n \in \{1,2, \ldots, N\}$ and let $\delta$ be any positive number which is smaller than the length of the interval $\alpha_n$. Let $a$ and $b$ be the endpoints of $\alpha_n$. By Frostman's construction, we have that $\tilde{\nu}_n([t, t+\delta]) \leq 4h(\delta)$ for any $t \in [a, b-\delta]$. It follows then that \begin{equation*}
    \int_a^{b-\delta} \tilde{\nu}_n([t, t+\delta]) dt \leq 4(b-\delta -a)h(\delta) \leq 4|\alpha_n|h(\delta).
\end{equation*}
The left-hand side of this equation can be re-written using Fubini's theorem:
\begin{gather*}
\int_a^{b-\delta} \tilde{\nu}_n([t, t+\delta]) dt = \int_a^{b - \delta} \int_{\alpha_n} 1_{t \leq s \leq t + \delta} \, d\tilde{\nu}_n(s) dt \\ = \int_{\alpha_n} \int_a^{b - \delta} 1_{t \leq s \leq t + \delta} \, dt \, d\tilde{\nu}_n(s) = \delta \tilde{\nu}_n(\alpha_n) = \delta c_n |\alpha_n|.
\end{gather*}
Let $\Delta$ be any subinterval of $\alpha_n$ of length $|\Delta| = \delta$. Then, from the above computations and \eqref{muaveragedef}, we get $$\nu(\Delta) = \delta c_n  \leq 4 h(|\Delta|).$$

Finally, let $\Delta$ be an arbitrary interval. Then $\Delta$ can be split up into at most three contiguous intervals $\Delta_1 \cup \Delta_2 \cup \Delta_3$, such that $\Delta_1$ and $\Delta_3$ are partially contained in some intervals $\alpha_{n_1}$ and $\alpha_{n_2}$, and $\Delta_2$ totally covers the intervals $\alpha_n$ which lie inbetween $\alpha_{n_1}$ and $\alpha_{n_2}$. For the interval $\Delta_2$ we have $\nu(\Delta_2) = \tilde{\nu}(\Delta_2) \leq 4 h(|\Delta_2|) \leq 4 h(|\Delta|)$, and we have above verified that $\nu(\Delta_1) \leq 4h(|\Delta_1|)$ and $\nu(\Delta_3) \leq 4h(|\Delta_3|)$. Thus $\nu(\Delta) \leq 12 h(|\Delta|)$. Upon division of the measure $\nu$ by 12, we finally obtain a measure satisfying $\nu(\Delta) \leq h(|\Delta|)$ for all intervals $\Delta \subset \T$, and $\nu(U) \geq \Mhd(U)/24$.

\end{proof}

\section{A real-variables construction}

\label{realvariablessection}

In this section we follow more or less the same path as Khrushchev does in \cite[Sections 4 and 10]{khrushchev1978problem}. We utilize the core-residual decomposition to handle the complications that arise due to our dealing with a completely general measurable subset $E$. 

Before going into the main construction, we single out the following technical lemma which appears in a slightly more complicated version in Khrushchev's work in \cite[Section 10]{khrushchev1978problem}. We include a proof for the readers convenience. 

\begin{lemma} \thlabel{khrushchevsetlemma}
Let $F$ be a measurable subset of an interval $I$. For any $\epsilon > 0$, there exists a closed subset $B$ of $I$ which is almost contained in $F$, and for any interval $\ell$ which is complementary to $B$ in $I$ it holds that $$|\ell \cap F| \leq \epsilon |\ell|.$$
\end{lemma}

\begin{proof}
Consider the set $X$ of points $x$ which are contained in $I \setminus F$ and which are Lebesgue density points of $I \setminus F$, in the sense that for the intervals $J(x, r)$ with center $x$ and length $r$, we have \begin{equation} \label{lebesguedensity}
\lim_{r \to 0} \frac{|J(x,r) \cap F|}{r} = \lim_{r \to 0} \frac{|J(x,r) \cap F|}{|J(x,r)|} = 0.
\end{equation} The set $X$ has full measure in $I \setminus F$. For each such $x$ we pick an interval $J(x) = J(x,r)$ with $r = r(x)$, such that the interval which has five times the length has small intersection with $F$: $$|J(x,5r) \cap F| \leq \epsilon |J(x,r)|.$$ 
Then the set $$B := I \setminus (\cup_{x \in X} J(x))$$ satisfies the required properties. Since $|X| = |I \setminus F|$, certainly $B$ is almost contained in $F$. Let $\ell$ be any interval complementary to $B$. Then $\ell$ is the union of those intervals $J(x)$ which are contained in it. By the classical $5r$-covering theorem (see \cite[page 23]{mattila1999geometry}) there exists a countable subfamily $\{J(x_n)\}_n$ of pairwise disjoint intervals which are contained in $\ell$ and such that $\ell$ is contained in the union $\cup_{n} J(x_n,5r)$. Hence $$|\ell \cap F| \leq \sum_n |J(x_n,5r) \cap F| \leq \epsilon \sum_n |J(x_n,r)| \leq \epsilon|\ell|,$$ where the last inequality follows by the disjointness of $J(x_n)$.
\end{proof}

The following construction is similar to one appearing in Khrushchev's work \cite[Section 4]{khrushchev1978problem}, but additional technical difficulties arise because of the presence of the set $\co(E)$. It is here that \thref{mainCRestimate} will be used to obtain a crucial lower bound on a certain mass. 

\begin{prop} \thlabel{fnsequenceprop}
Let $E$ be a measurable subset of $\T$. There exists a sequence of real-valued bounded measurable functions $\{f_n\}_n$ on $\T$ which satisfy the following properties:

\begin{enumerate}[(i)]
\item $\int_\Delta f_n \, d\m \leq 2 h(|\Delta|)$ for any interval $\Delta$ of $\T$,
\item $\int_\T f_n \, d\m = 0$,
\item $f_n(x) = 0$ for almost every $x \in \co(E)$,
\item $f_n(x) \leq 0$ for almost every $x \in \res(E)$,
\item $\lim_{n \to \infty} f_n(x) = -\infty$ uniformly for $x \in \res(E)$, up to a set of Lebesgue measure zero.
\end{enumerate}
\end{prop}

\begin{proof}
Fix a positive integer $n$ and recall the definitions \eqref{dyadicintervals} and \eqref{Dnintervals} of the $n$th generation dyadic intervals. We will show how to define $f_n$ piece by piece on the interior $I$ of any dyadic interval in $\mathcal{D}_n$. If $|I \cap \res(E)| = 0$, then we simply set $f_n \equiv 0$ on $I$. 

In the case $|I \cap \res(E)| > 0$, we apply \thref{khrushchevsetlemma} with data $F = I \cap E$ and $\epsilon = 1/2$. The lemma produces a closed set $B := H_I$ contained in $I$ and almost contained in $E$, with the property that if $\ell$ is an interval complementary to $H_I$ in $I$, then $|\ell \cap E| \leq |\ell|/2$. We next apply \thref{frostmanmeasureprop} to the open set $U = I \setminus H_I$. Let $$\tilde{\nu}_I = \sum_{n=1}^N c_n 1_{\alpha_n}$$ be the measure produced by that proposition, where $\alpha_n$ are open arcs complementary to $H_I$ and contained in $I$. This measure satisfies \begin{equation}
    \label{mutildemoc} \tilde{\nu}_I(\Delta) \leq h(|\Delta|)
\end{equation} for any interval $\Delta$ of $\T$, and \begin{equation} \label{mutildemass} \tilde{\nu}_I(I) \geq \Mhd(I \setminus H_I)/24.\end{equation} We produce a new smaller measure $\nu_I$ by restricting $\tilde{\nu}_I$ to the set $I \setminus E$. Thus, we set $\beta_n = \alpha_n \cap (I \setminus E)$ and let $$\nu_I := \sum_{n=1}^N c_n 1_{\beta_n}.$$ Note that \begin{gather*}
    \tilde{\nu}_I(I) = \sum_{n=1}^N c_n|\alpha_n| = \nu_I(I) + \sum_{n=1}^N c_n|\alpha_n \cap E| \\ \leq \nu_I(I) +  \Big(\sum_{n=1}^N c_n|\alpha_n|\Big)/2 = \nu_I(I) +  \tilde{\nu}_I(I)/2
\end{gather*} where we used that $|\alpha_n \cap E| \leq  |\alpha_n|/2$, a consequence of \thref{khrushchevsetlemma}. Hence we obtain that $$\nu_I(I) \geq \tilde{\nu}_I(I)/2.$$ Together with \eqref{mutildemass}, this gives $$\nu_I(I) \geq \Mhd(I \setminus H_I)/48,$$ and clearly $$\Mhd(I \setminus H_I)/48 \geq \Mho(I \setminus H_I)/48.$$
Further, using \thref{mainCRestimate} and \thref{mhinequalities} we obtain \begin{gather*} \Mho(I \setminus H_I)  = \Mho(I \setminus (H_I \cap \co(E))) \\ \geq \Mho(I \setminus \co(E)) \geq h(|I \setminus \co(E)|).\end{gather*} All in all, the string of inequalities gives us a lower estimate for the mass that $\nu_I$ places on $I$: \begin{equation}
    \label{finallowerestimate}
    \nu_I(I) \geq \frac{h(|I \setminus \co(E)|)}{48}.
\end{equation} 
We define $f_n$ on the interior of the interval $I$ by the formula \begin{equation}
    \label{fnIdef}
    f_n|I = \sum_{n=1}^N c_n 1_{\beta_n} - \frac{\nu_I(I)}{|I \cap \res(E)|}1_{I \cap \res(E)}.
\end{equation} For definiteness let us put $f_n$ to equal zero on the endpoints of $I$. The functions $f_n$ are clearly bounded. Our construction ensures that \begin{equation}
    \label{fnvanishI}\int_I f_n dm = 0
\end{equation} for any $I$ which is a dyadic arc of $n$th generation. Consequently, $$\int_\T f_n d\m = 0$$ also holds, so $(ii)$ above is satisfied. The properties $(iii)$ and $(iv)$ also hold. If $\Delta$ is an interval and it is contained in a $n$th generation dyadic interval $I$, then $$\int_\Delta f_n d\m = \nu_I(\Delta) - \sum_{n=1}^N c_n|\beta_n \cap \Delta| \leq \tilde{\nu}_I(\Delta) \leq h(|\Delta|).$$ If $\Delta$ is an interval, but it is not contained in any single such interval $I$, then we can split $\Delta$ up into contiguous intervals $\Delta_1 \cup \Delta_2 \cup \Delta_3$, such that $\Delta_1$ and $\Delta_3$ are contained in some dyadic intervals $I_1, I_2 \in \mathcal{D}_n$, and $\Delta_2$ is a union of the dyadic intervals in $\mathcal{D}_n$ which lie inbetween $I_1$ and $I_2$. Then, using \eqref{fnvanishI}, we can estimate $$\int_\Delta f_n d\m = \int_{\Delta_1} f_n d\m + \int_{\Delta_3} f_n d\m \leq 2 h(|\Delta|).$$ So $(i)$ also holds. Finally, consider $x \in \res(E)$. Using \eqref{finallowerestimate}, the obvious inequality $$|I \setminus \co(E)| \geq |I \cap \res(E)|$$ and the properties of $h$ postulated in Section \ref{hausdorffsec}, we have \begin{gather*}
    -f_n(x) = \frac{\nu_I(I)}{|I \cap \res(E)|} \geq \frac{h(|I \setminus \co(E)|}{48 |I \cap \res(E)|} \geq \frac{h(|I \cap \res(E)|)}{48 |I \cap \res(E)|} \geq \frac{h(|I|)}{48 |I|}.
\end{gather*} But $|I| = 2\pi 2^{-n}$, and consequently by the assumption that $\lim_{t \to 0^+} h(t)/t = \infty$ we conclude that $$\lim_{n \to \infty} -f_n(x) = \infty$$ for $x \in \res(E)$. This settles $(v)$, and the proof is complete.
\end{proof}

\section{A sequences which splits $\Po^2(\mu)$}

\label{splitsection}

We will now show that for a large family of measures $\mu$, the condition that $|\res(E)| > 0$ for the carrier set of $E$ of $\mu_\T$ implies that $\Po^2(\mu)$ contains an $L^2$-summand. We start by recalling some well-known results which will be used in the proof.

For a real-valued measure $\nu$ living on $\T$ we define its \textit{Herglotz integral} by the formula $$H_\nu(z) := \int_\T \frac{\zeta + z}{\zeta - z} d\nu(\zeta), \quad z \in \D.$$ The function $H_\nu$ is analytic in $\D$, and its real part $\Re(H_\nu)$ is the usual Poisson integral $P_\nu$ of $\nu$.  It is well-known that if $d\nu = fd\m$ is absolutely continuous, then $|H_\nu(z)| = \exp(|f(z)|)$ for almost every $z \in \T$, in the sense of the usual non-tangential boundary values.

The following estimate for the growth of $P_\nu$ in terms of smoothness of the measure $\nu$ (or more precisely, smoothness of its positive part) is also well-known (see, for instance, \cite[page 297]{havinbook}).

\begin{lemma} \thlabel{poissonintGrowth}
Let $h$ be a measure function satisfying the properties stated in Section \ref{hausdorffsec}. There exists an absolute constant $C > 0$ such that if $\nu(\Delta) \leq h(|\Delta|)$ for any interval $\Delta$ in $\T$, then 
\begin{equation}
    \label{Poissonestimate}
    P_\nu(z) = \Re(H_\nu(z)) \leq C \frac{h(1-|z|)}{1-|z|}, \quad z \in \D.
\end{equation}
\end{lemma}

We will also need the following approximation result which is a direct consequence of the classical Beurling-Wiener theorem. For a proof, see \cite{helsonbook}, for instance.

\begin{lemma} \thlabel{beurling-wiener}
Let $E$ be a measurable subset of $\T$ which satisfies $|E| < |\T|$. If a function $g \in L^2(E)$ is non-zero almost everywhere on a set $F \subset E$, then the norm-closure of the analytic polynomial multiples of $g$ in $L^2(E)$ contains the space $L^2(F)$.
\end{lemma}

We now state the main result of this section, and the most general result of the paper.

\begin{prop} \thlabel{p2musplittheorem}
Let $\mu = \mu_\D + \mu_\T$ be a positive finite Borel measure on the closed disk $\cD$, where $d\mu_\T = w \,d\m$ is absolutely continuous, and $\mu_\D$ is such that for all bounded analytic functions $p$, the estimate \begin{equation}
    \label{Muembedding} \int_{\D} |p(z)|^2 d\mu_\D(z)  \leq c_1 \sup_{z \in \D} |p(z)|^2 \exp\Bigg( -c_2 \frac{h(1-|z|)}{1-|z|}\Bigg)
\end{equation} holds for some measure function $h$ and some positive constants $c_1, c_2$. Let $E$ be a carrier set of $w$, i.e., a measurable subset of $\T$ such that $w = w1_E$ almost everywhere, and $1_{\res(E)} \cdot \mu_\T$ be the restriction of $\mu_\T$ to the set $\res(E)$. If $|\res(E)| > 0$, then the space $L^2(1_{\res(E)} \cdot \mu_\T)$ is contained in $\Po^2(\mu)$.
\end{prop}

Before going into the proof, we make two remarks. 

\begin{remark}
Firstly, note that in the case $h(t) = t \log(e/t)$, the inequality \eqref{Muembedding} is satisfied for some small $c_2 > 0$ for any measure $\mu_\D$ of the form \begin{equation} \label{weighbergman}
d\mu_\D = (1-|z|)^\alpha dA, \quad \alpha > -1.
\end{equation} Indeed, in that case the exponential expression in \eqref{Muembedding} equals a constant multiple of $(1-|z|)^{c_2}$, and the estimate holds for $\alpha - c_2 > -1$, since 
\begin{gather*}\int_\D |p(z)|^2 (1-|z|)^\alpha dA(z) \leq \sup_{z \in \D} |p(z)|^2(1-|z|)^{c_2} \int_\D (1-|z|)^{\alpha - c_2} dA(z) \\ \leq  c_1 \, \sup_{z \in \D} |p(z)|^2(1-|z|)^{c_2},
\end{gather*} where $$c_1 = \int_\D (1-|z|)^{\alpha - c_2} dA(z) < \infty.$$      
\end{remark} 
\begin{remark}
Secondly, in fact $L^2(\res(E))$ is the biggest $L^2$-summand contained in $\Po^2(\mu)$ in the case when $w$ is the characteristic function of $E$, $h(t) = t \log(e/t)$ and $\mu_\D$ is of the form \eqref{weighbergman}. We will elaborate on this in the next section, thus proving the first of the main theorems stated in the introduction.
\end{remark}

\begin{proof}[Proof of \thref{p2musplittheorem}]
We start by noting that if $g$ is a bounded analytic function in $\D$, then the function in $L^2(\mu)$ which coincides with $g$ in $\D$ and (almost everywhere) with the usual non-tangential boundary values of $g$ on $\mu_\T$, is actually contained in $\Po^2(\mu)$. This follows easily by dilating $g_r(z) := g(rz)$ and approximating the dilated functions uniformly in $\cD$ by analytic polynomials. 

Let $\{f_n\}_n$ be the sequence of real-valued bounded functions constructed in \thref{fnsequenceprop}. By property $(v)$ in that proposition, there exists an increasing sequence of integers $\{n_k\}_k$ such that \begin{equation} \label{fnOnResE}
\exp( f_{n_k}(\zeta)/k) \leq 1/k, \text{ for almost every  } \zeta \in \res(E).    
\end{equation}
We construct the bounded analytic functions \begin{equation}
    g_k(z) := \exp\big( H_{f_{n_k}}(z)/k\big), \quad z \in \D
\end{equation} By property $(i)$ in \thref{fnsequenceprop} and \thref{poissonintGrowth}, for sufficiently large $k$ the functions $g_k$ satisfy a growth estimate \begin{equation}
    \label{hkgrowthestimate} |g_k(z)|^2 \leq \exp\Big(\frac{4C}{k} \frac{h(1-|z|)}{1-|z|}\Big) \leq \exp\Big(c_2 \frac{h(1-|z|)}{1-|z|}\Big), \quad z \in \D.
\end{equation} The functions $g_k$ are bounded analytic functions in $\D$, so they are contained in $\Po^2(\mu)$ in the sense explained at the beginning of the proof. In terms of boundary values, the functions $g_k(z)$ converge uniformly almost everywhere to 0 on $\res(E)$, and by property $(iii)$ in \thref{fnsequenceprop}, we have $|g_k(z)| = 1$ almost everywhere on $\co(E)$, thus $|g_k(z)| \leq 1$ almost everywhere on $\mu_\T$. Keeping in mind the hypothesis \eqref{Muembedding} stated in the lemma, we conclude that $\{g_k\}_k$ is a norm-bounded sequence in $\Po^2(\mu)$. The estimate \eqref{hkgrowthestimate} shows also that the family $\{g_k\}_k$ is pointwise bounded uniformly on compact subsets of $\D$, hence we can extract a subsequence of $\{g_k\}_k$ which converges uniformly on compacts to some analytic function and also converges weakly to a function $g \in \Po^2(\mu)$. The function $g$ vanishes almost everywhere on $\res(E)$, by the uniform convergence of the functions $g_k$ to zero on $\res(E)$. By property $(ii)$ of \thref{fnsequenceprop}, the functions $g_k$ satisfy $g_k(0) = 1$, and so $g(0) = 1$. Next, the inequality in \eqref{hkgrowthestimate} shows that $$|g(z)| = \lim_{k \to \infty} |g_k(z)| \leq \lim_{k \to \infty} \exp\Big(\frac{4C}{k} \frac{h(1-|z|)}{1-|z|}\Big) = 1.$$ By the maximum modulus principle for analytic functions it follows now that $g(z) = 1$, for all $z \in \D$. 

Consider now $1-g \in \Po^2(\mu)$. This function lives only on $\mu_\T$, and it equals to $1$ almost everywhere on $\res(E)$. It follows that the entire subspace $\mathcal{M} \subseteq L^2(\mu_\T)$, which is the norm-closure of analytic polynomial multiples of $1-g$, is contained in $\Po^2(\mu)$. The linear mapping $T$ defined by $f \mapsto f\sqrt{w}$ is a unitary map between $L^2(\mu_\T)$ and $L^2(E)$, and it preserves subspaces invariant under multiplication by analytic polynomials. If $|\res(E)| > 0$, then certainly $|E| < |\T|$, so \thref{beurling-wiener} applies to tell us that $T\mathcal{M}$ equals $L^2(F)$, for some measurable subset $F$ which contains $\res(E)$. By applying the inverse mapping $T^{-1}$ to $L^2(F)$ we get that $L^2(1_{\res(E)} \cdot \mu_\T)$ is contained in $\Po^2(\mu)$.
\end{proof}

\section{Proof of Theorem A}

\label{biggestsummandsection}

We specialize to the case \begin{equation} \label{hform}
    h(t) = t\log(e/t).
\end{equation} For this choice, the $h$-Carleson sets are also known as Beurling-Carleson sets, or simply Carleson sets, as was noted in the introduction. We specialize also to the measures $\mu$ supported on the closed disk $\cD$ which are of the form \begin{equation}
    \label{muform}
    d\mu = d\mu_\D + d\mu_\T = (1-|z|)^{\alpha} dA + 1_E d\m,
\end{equation} where $\alpha > -1$ and $E$ is any measurable subset of $\T$. 

The lemma below is a direct consequence of, for instance, \cite[Corollary 5.3]{ConstrFamiliesSmoothCauchyTransforms}. It is also more or less known as a folklore theorem, and is a consequence of Khrushchev's work \cite[Section 3]{khrushchev1978problem} in slightly greater generality than we shall present it here. However, we have not found a precise reference for these more general statements, and so we present only what is necessary to complete the proof of our main theorem. The interested reader is invited to analyze Khrushchev's proof in \cite{khrushchev1978problem} to extract a more general statement which involves other measure functions $h$.

\begin{lemma} \thlabel{p2munonsplitBC}
If $E$ is a Beurling-Carleson set, then the space $\Po^2(\mu)$ with measure $\mu$ of the form \eqref{muform} will contain no characteristic functions of any subset of $\cD$, except the trivial ones: the zero function and the characteristic function of the carrier set, i.e., $1 \in \Po^2(\mu)$. 
\end{lemma}

We are now ready to prove our our main result, Theorem A of Section \ref{introsec}. We restate it for convenience. 

\begin{thm} \thlabel{theoremArestated}
Let $\mu$ be of the form \eqref{muform}. Assume that for $h$ of the form \eqref{hform} we have $|\res(E)| > 0$. Let $$d\mu_1 := d\mu -1_{\res(E)}d\m = d\mu_\D + 1_{\co(E)}d\m.$$ Then 
\begin{equation}\label{thomsondecompmu2}
 \Po^2(\mu) = L^2(\res(E)) \oplus \Po^2(\mu_1)
\end{equation}
where $\Po^2(\mu_1)$ contains no non-trivial characteristic functions.
\end{thm}

\begin{proof}
By \thref{p2musplittheorem} and the remarks appearing after its statement, the space $\Po^2(\mu)$ contains $L^2(\res(E))$. The orthogonal complement of this space in $\Po^2(\mu)$ is obviously $\Po^2(\mu_1)$, the point of contention is only if $\Po^2(\mu_1)$ contains a non-trivial characteristic function or not. It is easy to see that if $\Po^2(\mu_1)$ contains a non-trivial characteristic function, then it will contain the characteristic function $1_G$ of some measurable subset $G$ of $\co(E)$, since the functions in $\Po^2(\mu)$ are analytic in $\D$. 

Let $\co(E) = \cup_n {E_n}$, where the $E_n$ are Beurling-Carleson sets. Fix a sequence of analytic polynomials $\{p_n\}_n$ which converges to $1_G$ in the norm of $\Po^2(\mu_1)$. Then this sequence converges also in the norm induced by the smaller measures $d\mu_n := d\mu_\D + 1_{E_n} d\m$, to the characteristic function $1_{G \cap E_n}$. By \thref{p2munonsplitBC} the sequence must converge to $0$, i.e., $|G \cap E_n| = 0$. Since this is true for all $n$, $|G \cap \co(E)| = 0$, and so $|G| = 0$.
\end{proof}

We have thus completed the proof of Theorem A of the introduction.

\section{Proof of Theorem B}

\label{proofTBsec}

Theorem B stated in the introduction will follow easily from previous results found in \cite{ConstrFamiliesSmoothCauchyTransforms}, and from the next lemma.

\begin{lemma} \thlabel{lemmaForTheoremB}
Let $E$ be a measurable subset of $\T$, and $f \in L^2(E)$ be such that the Cauchy transform $C_f$ has a finite Dirichlet integral:
$$ \int_\D |C'_f|^2 dA < \infty.$$
Then $f$ vanishes almost everywhere on $\resLOG(E)$.
\end{lemma}

\begin{proof} A simple computation shows that if $g$ is an analytic function in $\D$, and $\{g_n\}_n$ is the sequence of its Taylor coefficients, then we have \begin{equation}
    \label{BergmanIsomorphism}
    \int_\D |g|^2 dA = \sum_{n=0}^\infty \frac{|g_n|^2}{n+1}.
\end{equation} Thus, if $f(\zeta) = \sum_{n \in \mathbb{Z}} f_n\zeta^n$ is the Fourier expansion of $f$, then the finiteness of the Dirichlet integral of $C_f$ is equivalent to \begin{equation}
    \sum_{n=0}^\infty (n+1)|f_n|^2 < \infty.
\end{equation}
Let $p(z) = \sum_{n=0}^N p_nz^n$ be any analytic polynomial. We can use Cauchy-Schwarz inequality to estimate  \begin{gather} 
    \Big\vert \sum_{n=0}^N f_n \conj{p_n} \Big\vert \leq \Big(\sum_{n=0}^\infty (n+1)|f_n|^2\Big)^{1/2}\Big(\sum_{n=0}^N \frac{|p_n|^2}{n+1} \Big)^{1/2} \nonumber \\ \label{est1}
    \lesssim \Big(\int_\D |C'_f|^2 dA \Big)^{1/2}\Big(\int_\D |p|^2 dA \Big)^{1/2}
\end{gather}
Now consider $h \in L^2(\resLOG(E))$. By Parseval's identity, and the fact that $f$ lives only on the set $E$, we have \begin{gather} 
    \int_E f\conj{h} \, d\m =  \int_E f\conj{(h-p)} \, d\m +  \int_E f\conj{p} \, d\m \nonumber \\ \label{est3} = \int_E f\conj{(h-p)} + \sum_{n=0}^N f_n \conj{p_n}.
\end{gather} 
By \thref{theoremArestated} for the case $\alpha = 0$, for any $\epsilon > 0$ we can find an analytic polynomial $p$ which satisfies \begin{equation} \label{pncond1}
    \int_{\D} |p|^2 dA < \epsilon
\end{equation} and 
\begin{equation}
    \label{pncond2}
    \int_E |p_n - h|^2 d\m < \epsilon.
\end{equation}
Consequently, a combination of \eqref{est1}, \eqref{est3}, \eqref{pncond1} and \eqref{pncond2} shows that $$\int_E f\conj{h} \, d\m = 0,$$ for any $h \in L^2(\resLOG(E))$. Thus $f$ must vanish almost everywhere on $\resLOG(E)$.
\end{proof}

\begin{proof}[Proof of Theorem B]
A set of functions in $S_+(E)$ which live only on $\coLOG(E)$ and are norm-dense in $L^2(\coLOG(E))$ has been constructed in \cite{ConstrFamiliesSmoothCauchyTransforms}. Conversely, \thref{lemmaForTheoremB} shows that if the Cauchy transform $C_f$ of any function $f \in L^2(E)$ has a finite Dirichlet integral, then $f$ lives on $\coLOG(E)$. The theorem is thus proved.
\end{proof}

\section{Proof of Theorem C}
\label{proofTCsec}

Let $D$ be the space of analytic functions in $\D$ with a  finite Dirichlet integral in \eqref{Dirichletintegral}. We shall give two proofs of Theorem C stated in the introduction.

Firstly, the validity of the claim in Theorem C follows immediately from \thref{p2musplittheorem} and the main result in \cite{DBRpapperAdem} which characterizes the norm-density of $D \cap \hb$ in $\hb$ in terms of certain operator-theoretic properties of the multiplication operator $M_z: \Po^2(\mu) \to \Po^2(\mu)$ for the measure \begin{equation} \label{muforDbrthm}
    d\mu = dA + (1-|b|^2)d\m,
\end{equation} the operator $M_z$ being one of multiplication by the coordinate function z: $$M_zf(z) = zf(z).$$ A necessary condition for density of $D \cap \hb$ is, according to \cite{DBRpapperAdem}, that $M_z:\Po^2(\mu) \to \Po^2(\mu)$ admits non non-trivial invariant subspace on which it acts as an isometry. Since the set \begin{equation}
    \label{EsetDBRproof}
    E := \{ \zeta \in \T : |b(\zeta)| < 1 \}
\end{equation} is the carrier set on $\T$ of the  weight $1-|b|^2$, \thref{p2musplittheorem} implies that if $|\resLOG(E)| > 0$, then $\Po^2(\mu)$ admits such a subspace. So in that case $D \cap \hb$ cannot be dense in $\hb$.

We can however give a proof of Theorem C which utilizes Theorem B and only the classical theory of $\hb$-spaces, passing around the abstract formulation in \cite{DBRpapperAdem}. Since we have already seen one proof, we give only a sketch of the next argument. For references to the facts about $\hb$s-spaces used below, one can consult \cite{hbspaces1fricainmashreghi} and \cite{hbspaces2fricainmashreghi}.

\begin{proof}[Sketch of an alternative proof of Theorem C]
Let $\Delta(\zeta) := \sqrt{1-|b(\zeta)|^2}, \zeta \in \T$. 
It is well-known that for each $f \in L^2(E)$, with $E$ given by \eqref{EsetDBRproof}, the Cauchy transform $C_{f\Delta}$ of the function $f\Delta$ (see \eqref{Cfeq}) is contained in $\hb$. The linear manifold of all such functions is usually denoted by $\hil(\conj{b})$ in the literature. If $b$ is a so-called extreme point of the unit ball of $H^\infty$ (which is the only case of interest for us), then it is known that the transform $f \mapsto  C_{f\Delta}$ is one-to-one on $L^2(E)$, and we have \begin{equation}
    \label{normineqDBRproof} \|C_{f\Delta}\|_{\hb} \leq \|C_{f\Delta}\|_{\hil(\conj{b})} := \|f\|_{L^2(E)}.
\end{equation} If $T_{\conj{b}}$ is the usual Toeplitz operator with the co-analytic symbol $\conj{b}$, then $T_{\conj{b}}: \hb \to \hil(\conj{b})$ is a contraction, in particular it is a bounded operator. It is also known that the space $D$ is invariant for any Toeplitz operator with a bounded co-analytic symbol.

Assume that $D \cap \hb$ is dense in $\hb$. Our task is then to show that the set $E$ given by \eqref{EsetDBRproof} is residual-free. Let $\{g_n\}_n$ be a sequence in $D \cap \hb$ which tends to the function $C_{1_E\Delta} \in \hb$ in the norm of $\hb$. By the invariance of $D$ for the operator $T_{\conj{b}}$, it follows that the sequence of $D$-functions $T_{\conj{b}}g_n$ tends to $T_{\conj{b}}C_\Delta = C_{\conj{b}\Delta}$ in the norm of $\hil(\conj{b})$. Since the functions $T_{\conj{b}}g_n$ are in $\hil(\conj{b})$, they are Cauchy transforms of the form $T_{\conj{b}}g_n = C_{f_n\Delta}$, for some $f_n \in L^2(E)$. By the norm-equality in \eqref{normineqDBRproof} we have that $f_n \to \conj{b}$ in $L^2(E)$, and so we can assume that $f_n \to \conj{b}$ pointwise almost everywhere on $E$ by passing to a subsequence. By Theorem B and the fact that $C_{f_n\Delta} = T_{\conj{b}}g_n \in D$ we must have $f_n \equiv 0$ on $\resLOG(E)$. But $\conj{b} \neq 0$ almost everywhere on $E$ (since $b$ is analytic), so we see that $\resLOG(E)$ must have Lebesgue measure zero, as claimed in the statement of Theorem C.
\end{proof}

\bibliographystyle{siam}
\bibliography{mybib}

\Addresses

\end{document}